\documentclass[12pt,tbtags]{amsart}

\usepackage{latexsym,enumerate}
\usepackage{amsmath,amsthm,amsopn,amstext,amscd,amsfonts,amssymb}
\usepackage{color}
\usepackage{marginnote}

\makeatletter
\def\th@definition{%
  \normalfont 
}
\def\th@plain{%
  \slshape 
}
\def\th@remark{%
  \normalfont 
  \thm@preskip\topsep
  \divide\thm@preskip\tw@
  \thm@postskip\thm@preskip
}
\makeatother

\numberwithin{equation}{section}
\theoremstyle{plain}

\newtheorem{theorem}{Theorem}

\newtheorem{proposition}[theorem]{Proposition}
\theoremstyle{remark}
\newtheorem{remark}[theorem]{Remark}
\theoremstyle{definition}
\newtheorem{definition}[theorem]{Definition}

\numberwithin{theorem}{section}

\def\Om{\Omega}

\def\R{\mathbb R}

\def\D{\nabla }

\newcommand{\disp}{\displaystyle}
\newcommand{\diw}{{\operatorname{div}}}
\newcommand{\meas}{\operatorname{meas}}
\newcommand{\aop}{\operatorname{\mathbf{a}}}
\newcommand{\med}{{\operatorname{med}}}
\newcommand{\sign}{{\operatorname{sign}}}

\def\m2{|\Omega | /2}

\begin{document}
\date{}

\title[]{Uniqueness  for Neumann problems for nonlinear elliptic equations}

\begin{abstract} In the present paper we prove uniqueness results for solutions to a class of Neumann boundary value problems 
whose prototype is
\begin{equation*} \label{pb0}
    \begin{cases}
        -\diw((1+|\D u|^2)^{(p-2)/2} \D u) -\diw (c(x)|u|^{p-2}u) =f   & \text{in}\ \Omega, \\[.1cm] 
      \left( (1+|\D u|^2)^{(p-2)/2} \D u+ c(x)|u|^{p-2}u \right)\cdot\underline n=0 & \text{on}\ \partial \Omega ,%
  \end{cases}%
\end{equation*}%
where $\Omega$ is a  bounded domain of $\R^{N}$,
$N\geq 2$, with Lipschitz boundary, $1< p< N$ ,
$\underline n$ is the outer 
unit normal to $\partial \Omega$, the datum $f$ belongs to $L^{(p^{*})'}(\Omega)$ or to $L^{1}(\Omega)$ and satisfies
the compatibility condition $\int_\Omega f \, dx=0$. Finally  the coefficient $c(x)$ 
belongs to an appropriate Lebesgue space.



\end{abstract}

\author{M.F. Betta}
\address{Maria Francesca Betta \hfill \break\indent Dipartimento di Ingegneria, \hfill\break\indent Universit\`a degli Studi di Napoli
Parthenope,\hfill\break\indent Centro Direzionale, Isola C4 80143 Napoli,
Italy}
\email{francesca.betta@uniparthenope.it}
\author{O. Guib\'e}
\address{Olivier Guib\'e \hfill \break\indent 
  Laboratoire de Math\'ematiques Rapha\"el Salem, \hfill\break\indent
UMR 6085 CNRS-Universit\'e de Rouen\hfill\break\indent
Avenue de l'Universit\'e, BP.12\hfill\break\indent
76801 Saint-\'Etienne-du-Rouvray, France
}
\email{olivier.guibe@univ-rouen.fr}
\author{A. Mercaldo}
\address{Anna Mercaldo \hfill \break\indent Dipartimento di Matematica e
Applicazioni ``R. Caccioppoli", \hfill\break\indent Universit\`a degli Studi di Napoli 
Federico II,\hfill\break\indent Complesso Monte S. Angelo, Via Cintia,
80126 Napoli, Italy}
\email{mercaldo@unina.it}
\maketitle

\section{Introduction}

In the present paper we prove uniqueness results for solutions to a class of Neumann boundary value problems whose prototype is
\begin{equation} \label{pb0}
    \begin{cases}
        -\diw((1+|\D u|^2)^{(p-2)/2} \D u)-\diw (c(x)|u|^{p-2}u) =f   & \text{in}\ \Omega, \\[.1cm] 
      \left((1+|\D u|^2)^{(p-2)/2} \D u+ c(x)|u|^{p-2}u \right)\cdot\underline n=0 & \text{on}\ \partial \Omega ,%
  \end{cases}%
\end{equation}%
where $\Omega$ is a  bounded domain of $\R^{N}$,
$N\geq 2$, with Lipschitz boundary, $1< p<N$ ,
$\underline n$ is the outer 
unit normal to $\partial \Omega$, the datum $f$ belongs to $L^{(p^{*})'}(\Omega)$, where $p^{*}=\frac{Np}{N-p}$, or to $L^{1}(\Omega)$ and satisfies
the compatibility condition $$\int_\Omega f \,dx=0.$$
Finally  the coefficient $c(x)$ 
belongs to an appropriate Lebesgue space which will be specified later.

The main difficulties in studying existence or uniqueness for this type of problems are due to the presence of a lower order term, the lower summability of the datum $f$ and the boundary Neumann conditions. 

\noindent  The existence for Neumann boundary value problems  with $L^1$-data
when $c=0$ has been  treated in various contests. In
\cite{AMST97}, \cite{Ciabr}, 
\cite{Droniou00}, \cite{DV} and \cite{Prignet97} the
existence of a distributional  solution which belongs to a suitable
Sobolev space  and which has null mean value is proved. Nevertheless
when $p$ is close to 1, i.e. $p\le 2-1/N$, the distributional solution
to problem \eqref{pb0} does not belong to a Sobolev space and in
general is not a summable function; this implies that its mean value
has not meaning and any existence result for distributional solution with null mean value cannot hold.  This difficulty is overcome in \cite{Rako} by
considering solutions $u$ which are not in $L^1(\Omega)$, but for
which $\Phi(u)$ is  in $L^1(\Omega)$, where
$\Phi(t)=\int_{0}^{t}\frac{ds}{(1+|s|)^\alpha}$ with appropriate $\alpha>1$. 

\noindent In \cite{ACMM} the case where 
both the datum $f$ and the domain $\Omega$ are not smooth enough is studied and  the existence and continuity with respect to the data of solutions whose  median is equal to zero is 
proved with a natural process of approximations and symmetrization techniques. 

\noindent We recall that the median of $u$ is defined by 
\begin{equation}\label{-520bis}
  {\rm med } (u) = \sup \{t\in \R : \meas\{u>t\} >     \meas(\Omega)/2\}\, .
\end{equation}
The existence for solutions having null median to problem \eqref{pb0}  when $c\neq 0$ are proved in \cite{BGM1}.

\noindent We explicitly remark that when the datum $f$ has a lower summability, i.e. it is just an $L^1$-function,
one has to give a meaning to the notion of solution; such a question has been faced already in the case where Dirichlet boundary conditions are prescribed, by introducing different notion of 
solutions (cf. \cite{BBGGPV}, \cite{Aglio}, \cite{LM}, \cite{murat94}  
). Such notion turn out to be equivalent, at least when the datum is an 
$L^{1}$-function. 

\noindent In the present paper, when $f\in L^1(\Omega)$, we refer to the so-called renormalized solutions
(see \cite{DMOP},  \cite{LM},  \cite{murat94}) whose
precise definition is recalled in Section 2. 

The main novelty of this article  is to prove uniqueness (up to additive
constants) results for renormalized solutions to problem
\eqref{pb0} having null median and whose existence has been proved  in \cite{BGM1}.

To our knowledge uniqueness results for problem \eqref{pb0} are new even in the
variational case, i.e. when $f$ belongs to $L^{(p^{*})'}(\Omega)$ and the usual notion of weak solution is considered.

\noindent  When 
$c(x)=0$ and $f$ is an element of the dual space of the Sobolev space
$W^{1,p}(\Omega)$, the existence and uniqueness (up to additive
constants) of weak solutions to problem \eqref{pb0} is consequence of
the classical theory of pseudo monotone operators (cfr. 
\cite{LL}, \cite{Lions}), while existence results for weak solutions to problem \eqref{pb0} when the lower order term appears  have been proved in \cite{BGM1}.

As pointed out we will prove different results according to the summability of $f$, i.e. $f\in
L^{(p^{*})'}(\Omega)$ or $f\in L^{1}(\Omega)$ and to the value of $p$, i.e. $
p\leq 2$ and $p\geq 2$. As far as $p$ is concerned such a difference is due to
the principal part of the operator, which we consider. Actually we assume that
 the principal part $-\diw(a(x,D u))$ is not degenerate when $p>2$, i.e. in the
model case $-\diw(a(x,\D u))=-\diw((1+|\D u|^2)^{(p-2)/2} \D u)$. But such an
assumption is not required when $p\leq 2$, that is for such values of $p$ we
prove uniqueness results for operators whose prototype is the so-called $p$-Laplace operator, $-\Delta_p u=-\diw(|\D
u|^{p-2}\D u)$.

Let us explain the main ideas of our results. To this aim let us consider the simpler case  of  weak solutions and $p=2$.
 When a Dirichlet boundary value problem is considered, following an idea
of Artola \cite{Ar86} (see also \cite{BGM92,ChiMi}), denoted by $u$ and $v$  two solutions, one can  use the
test function $T_k(u-v)$ and  obtain
 \begin{equation} \label{intr1}
\lim_{k\to 0}\frac{1}{k^2}\int_\Omega  |\nabla T_k(u-v)|^2\, dx  =0.  
\end{equation} 
Since $u$, $v\in H^1_0(\Omega)$,  Poincar\' e inequality implies that 
\begin{equation} \label{intr2}
\int_\Omega|\sign (u-v)|^2\, dx= \lim_{k\to 0}\frac{1}{k^2}\int_\Omega  | T_k(u-v)|^2\, dx  =0\,, 
\end{equation} 
from which one can deduce that $u=v $ a.e. in $\Omega$.
In contrast when we consider Neumann boundary conditions and two solutions $u,\,v\in H^1(\Omega)$ having null median, by using $T_k(u-v)$ we can prove equality \eqref{intr1}, but  Poincar\'e-Wirtinger inequality does not allow to get
\begin{equation*} 
\int_\Omega|\sign (u-v)|^2\, dx= 0  
\end{equation*} 
and therefore that $u=v $ a.e. in $\Omega$.
However  \eqref{intr1} and  Poincar\'e-Wirtinger inequality, allow to deduce
that
$u=v $ a.e. in $\Omega$ either $u>v $ a.e. in $\Omega$ either $u<v $ a.e. in
$\Omega$. Then we prove that $u>v $ a.e. in $\Omega$ or $u<v $ a.e.
in $\Omega$ leads to a contradiction: it  is done through a new test function
\begin{equation*}\label{natural}
w_{k,\delta}=\frac{T_k(u-v)}{k}\left ( \frac{T_\delta(u^+)}{\delta} -\frac{T_\delta(v^-)}{\delta}  \right)\,,
\end{equation*}
where $T_k$ denote the truncate function at height $k$ and 
$$
u^+=\max \{ 0\,, u \}\, , \qquad v^-=\max \{ 0\,, - v \}.
$$

Neumann problems have been studied by a different point of view in \cite{FM2}, \cite{ FM}, while existence or uniqueness results for Dirichlet boundary value problems for nonlinear elliptic equations with $L^1$-data are treated in \cite{BG1},\cite{BG2} and was  continued in various
contributions, including \cite{AM1}, \cite{BeGu},  \cite{BBGGPV},  \cite{BMMP1},  \cite{BMMP},  \cite{DMOP}, \cite{BDD} \cite{Aglio},   \cite{GM2},  \cite{GM1};
mixed boundary value problems have been also studied, for example, in \cite{BeGu}, \cite{Droniou00}.

The paper is organized as follows. In Section 2 we detail the assumptions and we
give the definition of a renormalized solution to \eqref{pb0}. Section 3 is
devoted to prove two uniqueness results for weak solutions when the datum is in the Lebesgue space $L^{(p^{*})'}(\Omega)$. In Section 4 we state our main results, Theorem
\ref{uniq_ren_1}, Theorem \ref{uniq_ren_2}, where we prove the uniqueness of a
renormalized solution to \eqref{pb0} when datum is a $L^1$ function.

\section{Assumptions and definitions}

Let us consider the following nonlinear elliptic  Neumann problem
\begin{equation} \label{pb}
\left\{
\begin{array}{lll}
-\mbox{div}\left( \aop\left( x, \nabla u\right)+ \Phi (x,u) \right) =f &  &
\text{in}\ \Omega, \\
 \left( \aop\left( x, \nabla u\right)+ \Phi (x,u) \right)\cdot\underline n=0& & \text{on}\ \partial \Omega ,%
\end{array}%
\right. 
\end{equation}%
where $\Omega $ is a bounded domain of $\mathbb{R}^{N}$, $N\ge
2$, having finite Lebesgue measure and Lipschitz boundary, $\underline
n$ is the outer unit normal to $\partial \Omega$.  We assume that $p$
is a real number such that  $1<p < N$.
 The function
$\aop\,:\,\Omega\times\R^N\mapsto \R^N$ is a Carath\'eodory 
function such that 
\begin{gather}
  \label{ell}
  \aop(x,\xi)\cdot\xi\geq \alpha |\xi|^p,\quad \alpha>0, \\
  \label{growth}
  |\aop(x,\xi)|\leq c[|\xi|^{p-1}+a_0(x)], \quad c>0,\quad a_0\in
  L^{p'}(\Omega),\quad a_0\geq 0,
\end{gather}
for almost every  $x\in\Omega$ and for every $\xi\in\R^N$.
 Moreover
    $\aop$ is strongly monotone, that is a constant $\beta>0$ exists such that
\begin{gather}
  \label{mon} (\aop(x,\xi)-\aop(x,\eta))\cdot(\xi-\eta) \geq \left\{
    \begin{aligned}
    \null &  \beta \frac{|\xi-\eta|^2}{(|\xi|+|\eta|)^{2-p}} & \text{
      if $1 <  p\leq 2$,} \\
    \null & \beta |\xi-\eta|^2 (1+|\xi|+|\eta|)^{p-2} &\text{ if
      $p\geq 2$,}
    \end{aligned}\right.
\end{gather}
for almost every $x\in \Omega$ and for every 
$\xi,\eta\in\R^N$, $\xi\neq \eta$.

\noindent
We assume that   $\Phi\,:\, \Omega\times \R\mapsto\R^N$
is a  Carath\'eodory function
which satisfies the following  ``growth condition''
 \begin{equation}
  \label{growthphi}
  |\Phi(x,s)|\leq c(x) (1+|s|)^{p-1},\quad c\in L^t(\Omega),\ c\geq 0,
 \end{equation}
 with  
 \begin{equation}
  \label{2n0bis} 
    t\ge\frac{N}{p-1}
    \end{equation}
for a.e. $x\in\Omega$ and for every  $s\in\R$. Moreover we assume that such function is locally
Lipschitz continuous with respect to the second variable, that is
\begin{equation}
  \label{lipphi}
   |\Phi(x,s)-\Phi(x,z)|\leq c(x) (1+|s|+|z|)^{\tau} |s-z|,
  \quad \tau\geq 0,
\end{equation}
for almost every $x\in\Omega$, for every $s$, $z\in\R$.
\par
Finally we assume that the datum $f$ is a measurable function in a Lebesgue space $L^r(\Omega)$, $1\le r\le +\infty$, which belongs to the dual space of the classical Sobolev space $W^{1,p}(\Om)$ or is just an $L^1-$ function. Moreover it satisfies  the compatibility condition
\begin{equation}
\int_{\Om}f\,dx=0.\label{comp}
\end{equation}

\par
As explained in the Introduction we deal with solutions whose median
is equal to zero.
Let us recall that if  $u$ is a measurable function, we denote  the
median of $u$ by
\begin{equation}
    \med (u) = \sup \left\{ t\in \R:\meas \{x\in \Om : u(x)>t\}
        >\frac{\meas(\Om)}{2} \right\}.
\end{equation}
Let us explicitely observe that if $\med(u)=0$ then
\begin{gather*}
    \meas \{x\in \Om : u(x)>0\}   \le\frac{\meas(\Om)}{2},
    \\
    \meas \{x\in \Om : u(x)<0\}  \le\frac{\meas(\Om)}{2} \,.
\end{gather*}
In this case a Poincar\'e-Wirtinger inequality holds (see e.g. \cite{Z}):
\begin{proposition}
If $u\in W^{1,p}(\Om) $, then
\begin{equation}
\|u-\med (u)\|_{L^p(\Omega)}\le C\|\D u \|_{(L^p(\Omega))^N} \label{poincare}
\end{equation}
where $C$ is a constant depending on $p$, $N$, $\Om$.
\end{proposition}

When the datum $f$ is not an element of the dual space of the classical Sobolev
space $W^{1,p}(\Om)$, the classical notion of weak
solution does not fit. We will refer to the notion of  renormalized solution to \eqref{pb} (see
\cite{DMOP,murat94} for elliptic equations with Dirichlet boundary conditions) which we give below.

In the whole paper, $T_{k}$, $k\ge 0$, denotes the truncation at height $
k$ that is $T_{k}(s)=\min(k,\max(s,-k))$, $\forall s\in\R$. 
\begin{definition}\label{defrenorm}
  A real function $u$ defined in $\Omega$ is a renormalized solution
  to \eqref{pb} if 
  \begin{gather}
    \label{def1}
    \text{ $u$ is measurable and finite almost everywhere in
      $\Omega$,}
    \\
    \label{def2}
    T_{k}(u)\in W^{1,p}(\Omega), \text{ for any $k>0$,}
    \\
    \label{def3}
    \lim_{n\rightarrow +\infty }\frac{1}{n} \int_{\{ x\in\Omega;\, |u(x)|<n\}}
    \aop(x,\nabla u) \nabla u \,dx = 0
  \end{gather}
  and if for every function $h$ belonging to $W^{1,\infty}(\R)$ with
  compact support  and for every $\varphi\in L^{\infty}(\Omega)\cap
  W^{1,p}(\Omega)$,
  we have
  \begin{multline}
      \int_{\Om} h(u) \aop (x, \D u) \D\varphi
  \,dx + \int_{\Om} h'(u) \aop (x, \D u) \D u   \varphi
  \,dx  \label{def4}\\ 
 +   \int_{\Om} h(u)\Phi (x, u)  \D\varphi \,dx + \int_{\Om} h'(u)\Phi (x, u)  \D u \varphi \,dx =
 \int_\Om f\varphi h(u) \,dx.
  \end{multline}
\end{definition}

\begin{remark} 
   A renormalized solution  is not  in general an
  $L^1_{loc}(\Omega)$-function and therefore it has not a
  distributional gradient. Condition \eqref{def2} allows to define a
  generalized gradient of $u$ according to Lemma 2.1 of \cite{BBGGPV},
  which asserts the existence of a unique measurable function $v$
  defined in $\Omega$ such that $\nabla T_k(u)=\chi_{\{|u|<k  \}}v$
  a.e. in $\Omega$, $\forall k>0$. This function $v$ is the
  generalized gradient of $u$ and it is denoted by $\nabla u$. 

Equality \eqref{def4} is formally obtained by using in \eqref{pb} the
test function $\varphi h(u)$ and  by taking into account Neumann boundary
conditions. Actually in a standard way one can  check that every term
in \eqref{def4} is well-defined under the structural assumptions on
the elliptic operator. 
\end{remark}

Let us recall Theorem 4.1 of \cite{BGM1}; under assumptions
\eqref{ell}-\eqref{comp} there exists at least one renormalized solution $u$
having null median of problem \eqref{pb}. Moreover any renormalized solution to
\eqref{pb} verifies the following proposition

\begin{proposition} \label{prop2.4}
  Under the assumptions  \eqref{ell}-\eqref{comp},  if $u$ denotes any renormalized solution  to problem \eqref{pb}, then
  \begin{equation}
    \label{ermk1}
    \lim_{n\rightarrow +\infty} \frac{1}{n} \int_{\Omega}
    |\Phi(x,u)|\,  |\nabla T_{n}(u)| dx = 0,
  \end{equation}
  \begin{equation}
    \label{ermk2}
   \int_{\Omega} \frac{ |\nabla u|^p}{(1+|u|)^{1+m}}\, dx\le C\, , \qquad \forall m>0,
  \end{equation}
where $C$ is a positive constant depending only on $m$, $f$, $\Omega$, $\alpha$
and $\Phi$
\begin{equation}
 \label{ermk4}
     |  u|^{p-1} \in L^q(\Omega), \forall \,1<q<\frac{N}{N-p},
  \end{equation}
  \begin{equation}
      \label{ermk3}
     |\nabla  u|^{p-1} \in L^q(\Omega), \forall \,1<q<\frac{N}{N-1}.
  \end{equation}

\end{proposition} 
 
 \noindent{\sl Sketch of the proof.} For the proof of  \eqref{ermk1} see Remark 2.4 of \cite{BGM1}. 

\noindent  The estimate \eqref{ermk2} is related to the Boccardo-Gallou\" et estimates \cite{BG1}, and it is obtained through a usual process. Indeed since $m>0$, 
$\disp \int_0^r\frac{ds}{(1+|s|)^{1+m}}\in L^\infty(\R)\cap \mathcal{C}^{1}(\R)$. Defining $h_n$ by
\begin{equation}  \label{h_n}
h_n(s)=
   \begin{cases}
       0    &\quad \text{if } |s|>2n, \\
\displaystyle\frac{2n-|s|}{n}     &\quad  \text{if } n<|s|\le 2n,\\
1     &\quad \text{ if } |s|\le n\,,
 \end{cases}
\end{equation}
we can use the renormalized formulation \eqref{def4} with $h=h_n$ and $\disp
\varphi=\int_0^{T_{2n}(u)}\frac{ds}{(1+|s|)^{1+m}}$. In view of \eqref{def3} and
\eqref{ermk1}, the growth condition \eqref{growthphi} on $\Phi$ allows one to
pass to the limit as $n \to +\infty$ and to obtain \eqref{ermk2}.


\noindent As far as \eqref{ermk4} and \eqref{ermk3} are concerned, it is sufficient to observe that
\eqref{def2}, \eqref{ermk2} and Poincar\'e-Wirtinger inequality imply (through
an approximation process) that $\disp
\int_0^u\frac{ds}{(1+|s|)^{\frac{1+m}{p}}}\in W^{1,p}(\Omega)$. Then Sobolev
embedding Theorem leads to
$$\forall m>0,\, |u|^{\frac{p-(1+m)}{p}}\in L^{\frac{Np}{N-p}}(\Omega)$$ 
which is equivalent to
$$\, |u|^{p-1}\in L^{q}(\Omega),\,\, \forall\, 1\le q< \frac{N}{N-p}.$$ 
Using again that $\disp
\int_0^u\frac{ds}{(1+|s|)^{\frac{1+m}{p}}}\in W^{1,p}(\Omega)$, \eqref{ermk4}
and H\"older inequality allow one to deduce \eqref{ermk3}.

 \section{Uniqueness results for weak solution}
\label{section_weak}

In this section we assume that the right-hand side $f$ is an element of the dual space $L^{(p^{*})'}(\Omega)$. In \cite{BGM1} an existence result for weak solution to problem \eqref{pb} having null median has been proved. Such a  weak solution $u$ is  a function such that
  \begin{gather*}
    u \in W^{1,p}(\Omega), \\
    \int_{\Omega} \aop(x,\nabla u)\nabla v dx +\int_{\Omega}
    \Phi(x,u) \nabla v dx =\int_{\Omega} fv  dx,
  \end{gather*}
  for any $v\in W^{1,p}(\Omega)$.

In this section we assume a suitable growth condition on $\Phi$, that is a bound on $\tau$ in \eqref{lipphi} is assumed and  the following assumption on the datum is made
\begin{equation}
  \label{dat}
  f\in  L^{(p^{*})'}(\Omega)\,. 
\end{equation}

Now we prove two uniqueness results depending on the values of $p$:

\begin{theorem} \label{uniq_weak1} Let $1<p< 2$. Assume that
 \eqref{ell}--\eqref{lipphi} 
  with 
 \begin{equation}\label{tau2}
 \tau\le p-1
     \end{equation}
  and \eqref{comp}, \eqref{dat} hold.
  If $u,v$ are two weak solutions to 
  problem \eqref{pb} having $\med(u)= \med(v)=0$, then $u=v$ a.e. in $\Omega$.
\end{theorem}
\smallskip

\begin{theorem} \label{uniq_weak2} Let $p\ge 2$. Assume that
 \eqref{ell}--\eqref{lipphi} 
  with 
 \begin{equation}\label{tau}
\tau\le \frac{Np}{N-p}\left(\frac 12-\frac{ 1}{t} \right),
 \end{equation}
 \begin{equation}\label{t}
 t\ge \max\left\{2, \frac{N }{p-1}\right\} 
 \end{equation}
  and \eqref{comp}, \eqref{dat} hold.
  If $u,v$ are two weak solutions to 
  problem \eqref{pb} having $\med(u)=\med(v)=0$, then $u=v$ a.e. in $\Omega$.
\end{theorem}
\smallskip

\begin{remark}\rm
We explicitely  observe that if $ 2\le p\le \frac{N+2}{2}$, we have uniqueness results under the assumption that $c$ belongs to $L^{\frac{N}{p-1}}(\Omega)$ assumption which guarantees the existence of a solution. If $ p> \frac{N+2}{2}$ the uniqueness result holds if $c$ belongs to $L^{2}(\Omega)$, which means that uniqueness result holds under a stronger assumption on the summability of $c$.
\end{remark} 
\smallskip

\begin{remark}\rm
Let us observe that the bounds on $\tau$ in the two theorems overlaps when $p=2$. 
\end{remark} 
\smallskip

\begin{proof}[Proof of Theorem \ref{uniq_weak1}]
Since for every fixed $k>0$, $T_k(u-v)\in W^{1,p}(\Omega)$, it can be used as test function in the equation satisfied by $u$ and in the equation satisfied by $v$. Then by subtracting the two equations, we get
\begin{gather}\label{iniz}
 \int_\Omega (\aop(x, \nabla u)- \aop(x, \nabla v))\cdot \nabla T_k(u-v)\, dx   \\
\notag \qquad \qquad     + \int_\Omega (\Phi(x,  u)- \Phi(x,  v))\cdot \nabla T_k(u-v)\, dx=0\,.
\end{gather}

We proceed by dividing the proof by steps.\smallskip

\noindent {\sl Step 1.} We prove that
\begin{equation} \label{step1fin}
\lim_{k\to 0}\frac{1}{k^p}\int_\Omega  |\nabla T_k(u-v)|^p\, dx  =0.  
\end{equation} 
By the assumptions on the strong monotonicity on the operator \eqref{mon} and the local Lipschitz condition on $\Phi$ \eqref{lipphi} with $\tau$ which satisfies \eqref{tau2}, we get
 \begin{gather}\label{iniz2b}
 \beta\int_\Omega \frac{|\nabla T_k(u-v)|^2}{(|\nabla u|+|\nabla v|)^{2-p} }\, dx  \\
 \notag     
\le k \int_\Omega c(x) (1+|u|+|v|)^\tau|\nabla T_k(u-v)|\, dx\,.
\end{gather}
The assumption on $\tau$ assures that the right-hand side of the previous inequality is finite. Moreover by H\"older inequality and assumption on $\tau$, we obtain
\begin{gather}\label{iniz3b}
\beta \int_\Omega \frac{|\nabla T_k(u-v)|^2}{(|\nabla u|+|\nabla v|)^{2-p} }\, dx  \qquad\qquad \\
\notag \qquad \qquad    \le k \left (  \int_{\{0<|u-v|<k\}}  c(x)^2 (1+|u|+|v|)^{2\tau} (|\nabla u|+|\nabla v|)^{2-p}\, dx  \right )^\frac12\\
\notag\times \left (    \int_\Omega \frac{|\nabla T_k(u-v)|^2}{(|\nabla u|+|\nabla v|)^{2-p} }\, dx   \right)^\frac12
\end{gather}
i.e.
\begin{gather}\label{iniz4b}
\frac{\beta^2}{k^{2}} \int_\Omega \frac{|\nabla T_k(u-v)|^2}{(|\nabla u|+|\nabla v|)^{2-p} }\, dx  \\
\notag \qquad \qquad    \le \int_{\{0<|u-v|<k\}}  c(x)^2 (1+|u|+|v|)^{2\tau} (|\nabla u|+|\nabla v|)^{2-p}\, dx .
\end{gather}
Since $\tau\le p-1=(1-\frac 1p-\frac{p-1}{N})\frac{Np}{N-p}$, H\"older
inequality assures that the integral in the right-hand side is finite.

\noindent  Since $\chi_{\{0<|u-v|<k\}}$ tends to 0 a.e. in $\Omega$ as $k$ goes
to 0, this implies 
\begin{equation}\label{step0b}
\lim_{k\to 0}\frac1{k^2} \int_\Omega \frac{|\nabla T_k(u-v)|^2}{(|\nabla u|+|\nabla v|)^{2-p} }\, dx   =0  .
\end{equation} 
Moreover by H\"older inequality we get
\begin{gather}\label{step1bb}
\frac{1}{k^p}\int_\Omega  |\nabla T_k(u-v)|^p\, dx  \\
\notag\qquad\qquad \le \left(\frac 1 {k^2}\int_\Omega \frac{|\nabla T_k(u-v)|^2}{(|\nabla u|+|\nabla v|)^{2-p} }\, dx  \right)^\frac{p}{2}
 \left(\int_\Omega (|\nabla u|+|\nabla v|)^{p} \, dx  \right)^{1-\frac{p}{2}}
\end{gather} 
which implies \eqref{step1fin} by \eqref{step0b}.
\medskip

\noindent {\sl Step 2.} We prove that either
\begin{equation*}  
   \begin{cases}
       u=v    &\quad \text{a.e. in } \Omega, \\
  u<v    &\quad \text{a.e. in } \Omega,\\
u>v    &\quad \text{a.e. in } \Omega.
 \end{cases}
\end{equation*}

Since $\frac{T_{k}(u-v)}{k}$ belongs to $W^{1,p}(\Omega)$ Poincar\'
e-Wirtinger inequality yields
\begin{gather}\label{PWTk}
\int_{\Omega}\left |  \frac{T_k(u-v)}{k}  - \med \left(\frac{T_k(u-v)}{k} \right ) \right |^p\, dx\\
\notag\le C
\int_{\Omega}
\left |  \frac {\nabla T_k(u-v)|}{k}
\right |^p\, dx  \,.
\end{gather} 
Therefore, by Step 1,  we deduce that
\begin{equation}\label{lim}
  \lim_{k\to 0}  \int_{\Omega}\left |
 \frac{T_k(u-v)}{k}  - \med \left(\frac{T_k(u-v)}{k} \right ) \right |^p\, dx=0. 
\end{equation} 
Since $\left|  \frac{T_k(u-v)}{k}  \right|\le 1$, we obtain
$$
\left| \med \left(\frac{T_k(u-v)}{k}\right )\right | \le 1\,, \,\,  k>0
$$
and, up to a subsequence, by  \eqref{lim}
$$
\lim_{k\to 0}  \med \left(\frac{T_k(u-v)}{k}\right )=\gamma
$$
 for a suitable constant $\gamma \in \R,\, |\gamma | \le 1$. On the other hand, we have
$$
\lim_{k\to 0} \frac{T_k(u-v)}{k}=\hbox{sign }(u-v)\,.
$$
Therefore, up to subsequence, by \eqref{lim} we get
$$
\int_\Omega |\hbox{sign }(u-v)-\gamma|^p\, dx=0
$$
which implies
$$
\gamma=0\,  \qquad \hbox{or}\qquad \gamma=-1\, \qquad \hbox{or}\qquad \gamma=1\,.
$$
This means that  either
\begin{equation*}  
   \begin{cases}
       u=v    &\quad \text{a.e. in } \Omega, \\
  u<v    &\quad \text{a.e. in } \Omega,\\
u>v    &\quad \text{a.e. in } \Omega.
 \end{cases}
\end{equation*}
\smallskip

\noindent {\sl Step 3.} We prove that $u<v, \, \text{a.e. in } \,\Omega\, \text{or }\,
 u>v\,,  \text{a.e. in }\Omega\,$ can not occur. 

We assume that 
\begin{equation}\label{uv}
u>v\,, \quad \hbox{a.e. in }\Omega\
\end{equation}
and  we prove that this yields a contradiction. The same arguments prove that $u<v$ a.e. in $\Omega$ 
can not be verified.

\noindent Since $\med(v)=0$,  $\disp \meas \{ x\in \Omega\, :\, v(x)< 0 \}\le \frac{\meas ( \Omega)}{2}$, then
\begin{equation}\label{meas}
 \meas \{ x\in \Omega\, :\, v(x)\ge 0 \}\ge \frac{\meas ( \Omega)}{2}\,.
\end{equation} 
On the other hand, we have
$$
\{ x\in \Omega\, :\, u(x)> 0 \}$$
$$
=\{ x\in \Omega\, :\, u(x)> 0\,, v(x)\ge 0 \}\cup \{ x\in \Omega\, :\, u(x)> 0\,, v(x)< 0 \}.
$$
Since we assume  \eqref{uv}, then we deduce 
$$
\{ x\in \Omega\, :\, u(x)> 0\,, v(x)\ge 0 \}=\{ x\in \Omega\, :\,  v(x)\ge 0 \}
$$
Therefore we get 
$$
\{ x\in \Omega\, :\, u(x)> 0 \}$$
$$=\{ x\in \Omega\, :\,  v(x)\ge 0 \}\cup \{ x\in \Omega\, :\, u(x)> 0,\, v(x)< 0 \}.
$$
Moreover, since \eqref{meas} holds true and since 
$$
\meas \{ x\in \Omega\, :\, u(x)> 0 \}\le \frac{\meas ( \Omega)}{2}\,,
$$
we conclude that
$$
\meas\{ x\in \Omega\, :\, u(x)> 0 \,, v(x)<0\}=0\,.
$$
This means that ``$u$ and $v$ have the same sign".

Now let us consider the test function
\begin{equation}\label{natural}
w_{k,\delta}=\frac{T_k(u-v)}{k}\left ( \frac{T_\delta(u^+)}{\delta} -\frac{T_\delta(v^-)}{\delta}  \right)\,,
\end{equation}
for fixed $k>0$, $\delta >0$, where 
$$
u^+=\max \{ 0\,, u \}\, , \qquad v^-=\max \{ 0\,, - v \}.
$$
Since $T_k(u-v)>0$ a.e. in $\Omega$, one can verify that
$$
\{ x\in \Omega\, :\, w_{k,\delta}(x)> 0 \}=\{ x\in \Omega\, :\, u(x)> 0 \}
$$
and
$$
\{ x\in \Omega\, :\, w_{k,\delta}(x)<0 \}=\{ x\in \Omega\, :\, v(x)< 0 \}.
$$
Moreover $T_\delta(u^+)\,, T_\delta(v^-)\in W^{1,p}(\Omega)$ and hence, since $\med\, (u)\,, \med\, (v)=0$, 
we conclude that
$$
\meas \{ x\in \Omega\, :\, w_{k,\delta}(x)> 0 \}\le \frac{\meas ( \Omega)}{2}\,,
$$
$$
\meas \{ x\in \Omega\, :\, w_{k,\delta}(x)< 0 \}\le \frac{\meas ( \Omega)}{2}\,,$$
this means 
$$\med\, (w_{k,\delta})=0\,.
$$
Therefore by Poincar\' e-Wirtinger inequality we deduce
\begin{equation}\label{PWw}
\int_\Omega|w_{k,\delta}|^p\, dx\le C \int_{\Omega} |\nabla w_{k,\delta}|^p\, dx\,.
\end{equation}
We now evaluate the gradient of $w_{k,\delta}$,
\begin{gather}\label{gradw}
\nabla w_{k,\delta}=\frac{\nabla T_k(u-v)}{k} \left ( \frac{T_\delta(u^+)}{\delta} -\frac{T_\delta(v^-)}{\delta}  \right)\\
\notag \qquad \qquad+
\frac{T_k(u-v)}{k}\left (\frac{\nabla u}{\delta} 
\chi_{\{\, 0<u<\delta  \} } + \frac{\nabla v}{\delta} \chi_{\{\, -\delta<v<0  \} }  \right )\, \text{ a.e. in } \Omega.
\end{gather}
Since $u$ and $v$ ``have the same sign", then, for every fixed $k>0$, it results
$$
0<\frac{T_k(u-v)}{k\delta} \chi_{\{\, 0<u<\delta  \} }\le \frac 1k\chi_{\{\, 0<u<\delta  \} }\,,
$$
$$
0<\frac{T_k(u-v)}{k\delta} \chi_{\{\, -\delta<v<0  \} }\le \frac 1k\chi_{\{\, -\delta<v<0  \} }\,,
$$
then for fixed $k>0$, we have 
$$
\lim_{\delta\to 0}\frac{T_k(u-v)}{k\delta} \chi_{\{\, 0<u<\delta  \} }=0 \qquad \hbox{a.e. in }\Omega\,.
$$
$$
\lim_{\delta\to 0}\frac{T_k(u-v)}{k\delta} \chi_{\{\, -\delta<v<0  \} }=0 \qquad \hbox{a.e. in }\Omega\,.
$$
Moreover we have also
$$
\left |   \frac{T_k(u-v)}{k\delta} \nabla u \chi_{\{\, 0<u<\delta  \} }  \right |\le \frac 1k |\nabla u|\chi_{\{\, 0<u<\delta  \} }\,,
$$
$$
\left |   \frac{T_k(u-v)}{k\delta} \nabla v \chi_{\{\, -\delta<v<0  \} }  \right |\le \frac 1k |\nabla v|\chi_{\{\, -\delta<v<0  \}} \,,
$$
and since $|\nabla u|, |\nabla v|\in L^p(\Omega)$,  we can apply Lebesgue dominated convergence Theorem, i.e.
$$
\lim_{\delta \to 0} \int_\Omega|\nabla w_{k,\delta}|^p\, dx=
\int_{\Omega}\left |   \frac{\nabla T_k(u-v)}{k} \left(\chi_{\{\, u>0  \}}  -
\chi_{\{\, v<0  \}} \right)      \right |^p\, dx\,.
$$
Since 
$$
\int_{\Omega}\left |   \frac{\nabla T_k(u-v)}{k} \left(\chi_{\{\, u>0  \}}  -
\chi_{\{\, v<0  \}} \right)      \right |^p\, dx\le 
\int_{\Omega}\left |   \frac{\nabla T_k(u-v)}{k} 
 \right |^p\, dx,
$$
by Step 1, we conclude that 
$$
\lim_{k\to 0}\lim_{\delta \to 0} \int_\Omega|\nabla w_{k,\delta}|^p\, dx=0\,.
$$
Now we can pass to the limit in \eqref{PWw} as $\delta \rightarrow 0$ first  and then as $k\rightarrow 0$ and we get
$$
\int_\Omega \left| \hbox{sign }(u-v) \left ( \chi_{\{\, u>0  \}} - \chi_{\{\,
      v<0  \}}\right)\right |^p\, dx=\int_{\Omega} |\sign(u)|^{p}dx = 0\,.
$$
We deduce that $\chi_{\{\, u>0  \}}=\chi_{\{\, v<0  \}}$ a.e. in $\Omega$; this yields a contradiction since 
we have proved that $u$ and $v$ have the same sign.

The same arguments yield that we can not have $u<v$ a.e. in $\Omega$. The conclusion follows.
\end{proof}

\begin{proof}[Proof of Theorem \ref{uniq_weak2}]

As in the previous proof we arrive to equality \eqref{iniz} and we divide the
proof by 3 steps.

\noindent {\sl Step 1.}  We prove that
\begin{equation}\label{step1}
\lim_{k\to 0}\frac1{k^2} \int_{\Omega}|\nabla T_k(u-v)|^2\, dx =0 .
\end{equation} 
By the assumptions on the strong monotonicity on the operator \eqref{mon} and the local Lipschitz condition on $\Phi$ \eqref{lipphi} with $\tau$ which satisfies \eqref{tau}, we get
\begin{gather}\label{iniz2}
\beta \int_\Omega (1+|\nabla u|+|\nabla v|)^{p-2} |\nabla T_k(u-v)|^2\, dx  \\
\notag \qquad \qquad     \le k \int_\Omega c(x) (1+|u|+|v|)^\tau|\nabla T_k(u-v)|\, dx\,.
\end{gather}
Since $p\ge 2$,   $T_k(u-v)$ belongs to $W^{1,2}(\Omega)$. Then by H\"older inequality and assumption on $\tau$, we obtain
\begin{gather}\label{iniz3}
\beta \int_\Omega |\nabla T_k(u-v)|^2\, dx   \\
\notag \qquad \qquad    \le  k\|c\|_{L^t(\{   0<|u-v|<k \})}
\|1+|u|+|v|\|^{\tau}_{L^{p^*}}\|\nabla T_k(u-v)\|_{L^2}\, dx\,,
\end{gather}
i.e.
\begin{equation}\label{iniz4}
\frac{\beta^2}{k^{2}} \int_\Omega |\nabla T_k(u-v)|^2\, dx   
  \le  \|c\|^2_{L^t(\{   0<|u-v|<k \})}
\|1+|u|+|v|\|^{2\tau}_{L^{p^*}} 
\,.
\end{equation}
Since $\chi_{\{   x\,: \, 0<|u-v|<k \}}\rightarrow 0$ a.e. in $\Omega$, Lebesgue
dominated convergence theorem implies that \eqref{step1} holds.
\smallskip

\noindent {\sl Step 2.} We prove that either
\begin{equation*}  
   \begin{cases}
       u=v    &\quad \text{a.e. in } \Omega, \\
  u<v    &\quad \text{a.e. in } \Omega,\\
u>v    &\quad \text{a.e. in } \Omega.
 \end{cases}
\end{equation*}
By Poincar\' e-Wirtinger inequality, we get
\begin{gather}\label{PWTk_bis}
\int_{\Omega}\left |  \frac{T_k(u-v)}{k}  - \med \left(\frac{T_k(u-v)}{k} \right ) \right |^2\, dx\\
\nonumber \le C
\int_{\Omega}
\left |  \frac {\nabla T_k(u-v)|}{k}
\right |^2\, dx  \,.
\end{gather} 
Therefore, by Step 1. we deduce that
\begin{equation}\label{lim_bis}
\lim_{k\to 0}  \int_{\Omega}\left |  \frac{T_k(u-v)}{k}  - \med \left(\frac{T_k(u-v)}{k} \right ) \right |^2\, dx=0
\end{equation} 
Since $\left|  \frac{T_k(u-v)}{k}  \right|\le 1$, we obtain
$$
\left| \med \left(\frac{T_k(u-v)}{k}\right )\right | \le 1\,, 
$$
and, up to a subsequence, 
$$
\lim_{k\to 0}  \med \left(\frac{T_k(u-v)}{k}\right )=\gamma
$$
for a suitable constant $\gamma \in \R,\, |\gamma|\le 1$. On the other hand, we have
$$
\lim_{k\to 0} \frac{T_k(u-v)}{k}=\hbox{sign }(u-v)\,,
$$
Therefore, up to subsequence, we get
$$
\int_\Omega |\hbox{sign }(u-v)-\gamma|^2\, dx=0
$$
which implies
$$
\gamma=0\,  \qquad \hbox{or}\qquad \gamma=-1\, \qquad \hbox{or}\qquad \gamma=1\,.
$$
This means that  either
$$
u=v\,, \quad \text{a.e. in }\Omega\,\quad \text{or}  \quad u<v\,, \quad \text{a.e. in }\Omega\,\quad \text{or}
\quad u>v\,, \quad \text{a.e. in }\Omega\,.
$$
\smallskip
\noindent {\sl Step 3.} Arguing as in Step 3 of the previous theorem, we prove that the last two possibilities can not occur.
Then conclusion follows.
\end{proof}

\begin{remark}
In \cite{BGM1} we estabilished the existence of a weak solution when
$\aop(x,\xi)$ is replaced by a Leray-Lions operator $\aop(x,r, \xi)$ which
depends on $x$, $r$ and $\xi$ and verifies the standard conditions (see
\cite{LL}). In the Dirichlet case and  $1<p\le 2$ it is well known (see
\cite{BGM92} \cite{ChiMi}) that under suitable assumptions on $\aop(x,r, \xi)$
the weak solution is unique.

\noindent In view of the proofs of Theorem \ref{uniq_weak1} and Theorem
\ref{uniq_weak2} it is possible to obtain the uniqueness of the weak solution
having null median of the problem  
\begin{equation} \label{pbcompleto}
\left\{
\begin{array}{lll}
-\mbox{div}\left( \aop\left( x, u, \nabla u\right)+ \Phi (x,u) \right) =f &  &
\text{in}\ \Omega, \\
 \left( \aop\left( x, u, \nabla u\right)+ \Phi (x,u) \right)\cdot\underline n=0& & \text{on}\ \partial \Omega.%
\end{array}%
\right. 
\end{equation}%
If we assume that $\aop(x,r, \xi)$
 is a Carath\'eodory 
function which verifies
\begin{gather}
  \label{ell_comp}
  \aop(x,s, \xi)\cdot\xi\geq \alpha |\xi|^p,\quad \alpha>0, \\
  \label{growth_comp}
  |\aop(x,s, \xi)|\leq c_1[|\xi|^{p-1}+|s|^{p-1}+a_0(x)], \, \\
  \notag c_1>0,\, a_0\in
  L^{p'}(\Omega),\, a_0\geq 0,
\end{gather}
 \begin{gather}
  \label{mon_comp} (\aop(x,s,\xi)-\aop(x,s,\eta))\cdot(\xi-\eta) \geq \left\{
    \begin{aligned}
    \null &  \beta \frac{|\xi-\eta|^2}{(|\xi|+|\eta|)^{2-p}} & \text{
      if $1\leq p\leq 2$,} \\
    \null & \beta |\xi-\eta|^2 (1+|\xi|+|\eta|)^{p-2} &\text{ if
      $p\geq 2$,}
    \end{aligned}\right.
\end{gather}
and moreover  $\aop(x,r, \xi)$ satisfies a Lipschitz condition   with respect to $r$
\begin{gather}\label{lips_a}
  |\aop(x,s, \xi)- \aop(x,r, \xi)|\leq c_2|s-r|(|\xi|^{p-1}+|s|^{p-1}+|r|^{p-1} + h(x)), \\
  \notag \, c_2>0,\, h\in
  L^{p'}(\Omega),\, h\geq 0,
\end{gather}
for almost every $x\in\Omega$, $s\in \R$ and for every $\xi\in\R^N$, then
Theorem~\ref{uniq_weak1} and Theorem~\ref{uniq_weak2} hold true. Indeed the
methods developped in \cite{BGM92} allow one to prove {\sl Step 1} in Theorem
\ref{uniq_weak1} namely
\begin{equation*} \label{step1_w_comp}
\lim_{k\to 0}\frac{1}{k^p}\int_\Omega  {|\nabla T_k(u-v)|^p}\, dx  =0.  
\end{equation*} 
and  {\sl Step 1} in  Theorem \ref{uniq_weak2}  namely 
\begin{equation*} \label{step1_w_comp}
\lim_{k\to 0}\frac{1}{k^2}\int_\Omega  {|\nabla T_k(u-v)|^2}\, dx  =0.  
\end{equation*} 
In both cases the  {\sl Step 2} and  {\sl Step 3} remain unchanged.
\end{remark}
\medskip

\begin{remark} 
In \cite{BGM1} and in the present paper we have chosen to deal with solutions to
\eqref{pb} with null median value instead of null mean value. As explained in
Introduction this choice allows one to consider solution to \eqref{pb} for $f\in
L^1(\Omega)$ even if the solution $u$ does not belong to $ L^1(\Omega)$. When
$f\in L^{(p^{*})'}(\Omega)$ a simply examination of the proof of \cite{BGM1} leads
to the existence of solutions to \eqref{pb} such that $\disp \int_\Omega u\, dx=0$. Assuming
that \eqref{ell}--\eqref{lipphi} are in force similar arguments to the one developped in
the proof of Theorem~\ref{uniq_weak1} and Theorem~\ref{uniq_weak2} yield the
uniqueness of solution to \eqref {pb} having a null mean value. Let us explain
briefly the case $p=2$. {\sl Step 1} remains unchanged so that if $u$ and $v$
are two solutions of \eqref{pb} then we have
\begin{equation*} 
\lim_{k\to 0}\frac1{k^2} \int_{\Omega}|\nabla T_k(u-v)|^2\, dx =0  .
\end{equation*} 
Poincar\' e-Wirtinger inequality leads to
\begin{gather*}
\lim_{k\to 0}\int_{\Omega}\left |  \frac{T_k(u-v)}{k}  - \frac{1}{|\Omega|} \int_{\Omega}   \frac{T_k(u-v)}{k} dy  \right |^2\, dx=0.
\end{gather*} 
so that, up to subsequence, there exists $\gamma\in[-1,1]$ such that 
$$
\int_\Omega |\hbox{sign }(u-v)-\gamma|^2\, dx=0.
$$
As in {\sl Step 2}
$$
\gamma=0\,  \qquad \hbox{or}\qquad \gamma=-1\, \qquad \hbox{or}\qquad \gamma=1\,
$$
and 
\begin{equation*}  
   \begin{cases}
       u=v    &\quad \text{a.e. in } \Omega, \\
  u<v    &\quad \text{a.e. in } \Omega,\\
u>v    &\quad \text{a.e. in } \Omega.
 \end{cases}
\end{equation*}
We now show  that $u<v$  {a.e. in } $\Omega$ {or}
$u>v$ {a.e. in }$\Omega\,$ can not occur. 
The method is similar to {\sl Step 3} of the proof of Theorem \ref{uniq_weak1}:
\begin{equation*} 
w_{k,\delta}=\frac{T_k(u-v)}{k}\left ( \frac{T_\delta(u^+)}{\delta} -\frac{T_\delta(v^-)}{\delta}  \right)\,,
\end{equation*}
belongs to $H^1(\Omega) $ while 
$$\lim_{k\to0 }\lim_{\delta \to0 }\|\nabla w_{k,\delta}\|_{L^2(\Omega)}=0.$$
Poincar\' e-Wirtinger inequality yields
\[
  \lim_{k\to0 }\lim_{\delta \to0 } \int_\Omega \left |w_{k,\delta}  -
    \frac{1}{|\Omega|} \int_\Omega  w_{k,\delta} dy \right|^2\, dx=0.
\]
In the case $u>v \quad \text{a.e. in } \Omega$, the Lebesque dominated Theorem allows one to conclude that 
\[
\int_\Omega \left| \sign(u) -  \frac{1}{|\Omega|}\int_\Omega \sign(u)   dy \right |^2\, dx=0\,,
\]
and then  $u$ has a constant sign. Recalling that    $\int_\Omega u\, dx=\int_\Omega v\, dx=0$ gives a contradiction.
Therefore $u=v \quad \text{a.e. in } \Omega$.
\end{remark}


\section{Uniqueness result for  renormalized solution}

In this section we prove  the uniqueness  of the renormalized solution to problem \eqref{pb}, when the following assumption on datum is made
\begin{equation}
  \label{dat_L1}
  f\in L^{1}(\Omega)\,. 
\end{equation}

As in Section \ref{section_weak}  we state two uniqueness theorems  depending on the values of $p$:

\begin{theorem} \label{uniq_ren_1} Let $1<p< 2$. Assume that
 \eqref{ell}--\eqref{lipphi} 
  with 
 \begin{equation}\label{tau2bis}
 \tau\le p-\frac32 +\left[\frac{p-1}{N}-\frac1t \right] \frac{N(p-1)}{N-p}
     \end{equation}
  and \eqref{comp}, \eqref{dat_L1} hold.
  If $u,v$ are two renormalized  solutions to 
  problem \eqref{pb} having $\med(u)= \med(v)=0$, then $u=v$ a.e. in $\Omega$.
\end{theorem}
\medskip

\begin{theorem} \label{uniq_ren_2} Let $p\ge 2$. Assume that
 \eqref{ell}--\eqref{lipphi} 
  with 
 \begin{equation}\label{taubis}
 \tau\le \frac{N(p-1)}{N-p}\left(\frac12 -\frac1t\right)
 \end{equation}
 \begin{equation}\label{t_ren2}  t\ge \max\left\{2, \frac{N }{p-1}\right\}
  \end{equation}
  and \eqref{comp}, \eqref{dat_L1} hold.
  If $u,v$ are two renormalized solutions to 
  problem \eqref{pb} having $\med(u)= \med(v)=0$, then $u=v$ a.e. in $\Omega$.
\end{theorem}
\medskip

\begin{proof}[Proof of Theorem \ref{uniq_ren_1}]

Let $u$ and $v$ be two renormalized solutions to \eqref{pb}. 
Let $h_{n}$ defined by $\eqref{h_n}$.
Since for any $k>0$, $\disp h_{n}(u) {T_{k}(u-v)}=
h_{n}(v){T_{k}(T_{2n}(u)-T_{2n+k}(u))} \in L ^\infty(\Omega)\cap W^{1,p
}(\Omega)$, we can use $h=h_n(u)$ and $\varphi =h_{n}(v){T_{k}(u-v)}$ in
$\eqref{def4}$ written in $u$, and we can use $h=h_n(v)$ and $\varphi
=h_{n}(u){T_{k}(u-v)} $ in $\eqref{def4}$ written in $v$. By substracting the
two equations, we get
\begin{gather}\label{iniz_ren}
 \int_\Omega h_n(u)h_n(v)(\aop(x, \nabla u)- \aop(x, \nabla v))\cdot \nabla T_k(u-v)\, dx   \\ 
\notag    +  \int_\Omega h_n(u)h_n(v)(\Phi(x,  u)- \Phi(x,  v))\cdot \nabla T_k(u-v)\, dx\\
\notag +\int_\Omega h'_n(u)h_n(v)T_k(u-v)(a(x, \nabla u)+ \Phi(x,  u)- a(x, \nabla v)-\Phi(x,  v))\cdot \nabla u\, dx   \\
\notag    +  \int_\Omega h_n(u)h'_n(v)T_k(u-v)(a(x, \nabla u)+ \Phi(x,  u)- a(x, \nabla v)-\Phi(x,  v))\cdot \nabla v\, dx=0\,.
\end{gather}
 We proceed by dividing the proof into 3 steps.
\medskip

\noindent {\sl Step 1.} By passing to the limit in \eqref{iniz_ren} first as
$n\to +\infty$, then as $k\to 0$ this step is to devoted to prove that
\begin{equation} \label{step1fin_ren}
\lim_{k\to 0}\frac{1}{k^2}\int_\Omega  \frac{|\nabla T_k(u-v)|^2}{(|\nabla u|+|\nabla v|)^{2-p }}\, dx  =0.  
\end{equation} 

We first study the behaviour of the last two integrals in \eqref{iniz_ren}  as
$n$ goes to $+\infty$ by showing
\begin{equation}\label{step1_4_ren}
\lim_{n\to+\infty} \int_\Omega h'_n(u)h_n(v)T_k(u-v)(\aop(x, \nabla u)+ \Phi(x,  u)- \aop(x, \nabla v)-\Phi(x,  v))\cdot \nabla u\, dx =0
\end{equation}
and by symmetry with respect to $u$ and $v$
\begin{equation}\label{step1_5_ren}
\lim_{n\to+\infty} \int_\Omega h'_n(v)h_n(u)T_k(u-v)(\aop(x, \nabla u)+ \Phi(x,  u)- \aop(x, \nabla v)-\Phi(x,  v))\cdot \nabla v\, dx =0.
\end{equation}

By \eqref{def3} of Definition \ref{defrenorm} and \eqref{ermk1} of Proposition 2.4, we get
\begin{equation}\label{step1_1_ren}
\lim_{n\to+\infty}\int_\Omega h'_n(u)h_n(v)T_k(u-v)\aop(x, \nabla u)\cdot\nabla u\, dx=0,
\end{equation}
\begin{equation}\label{step1_1bis_ren}
\lim_{n\to+\infty}\int_\Omega h'_n(u)h_n(v)T_k(u-v)\Phi(x,  u)\cdot \nabla u\, dx=0.
\end{equation}

By  assumption \eqref{growth} and H\"older inequality we have
\begin{gather*}
\left | \int_\Omega h'_n(u)h_n(v)T_k(u-v)\aop(x, \nabla v)\cdot\nabla u\, dx\right |\\
\le \frac{ck}{n}\left(\int_{\{|v|\le2n\}}(|a_0(x)| +|\nabla v|^{p-1})^{\frac{p}{p-1}}\, dx\right)^{\frac{p-1}{p}} 
\left(\int_{\{|u|\le2n\}}|\nabla u|^{p} \, dx\right)^{\frac{ 1}{p}}. \notag
\end{gather*}
Using  \eqref{ell}  and \eqref{def3} we deduce that
\begin{equation}\label{cond_ren}
\lim_{n\to+\infty}  \frac1n \int_{\{|u|\le2n\}}|\nabla u|^{p} \, dx =0.
\end{equation}
Therefore recalling that $a_0\in L^{p'}(\Omega)$ we conclude that
\begin{equation}\label{step1_2_ren}
\lim_{n\to+\infty}\left| \int_\Omega h'_n(u)h_n(v)T_k(u-v)\aop(x, \nabla v)\cdot\nabla u\, dx\right|=0.
\end{equation}
To prove that \eqref{step1_4_ren} holds it remains to control  $\int_\Omega
h'_n(u)h_n(v)T_k(u-v)\Phi (x,  v)\cdot\nabla u\, dx$. 
By assumption \eqref{growthphi} and H\"older inequality we have
\begin{gather*}
\left | \int_\Omega h'_n(u)h_n(v)T_k(u-v)\Phi (x,  v)\cdot\nabla u\, dx\right |\\
\le k \left(\frac1n\int_{\{|v|\le2n\}}|c(x)|^{\frac{p}{p-1}}(1+| v|)^{{p}} \, dx\right)^{\frac{p-1}{p}} 
\left(\frac1n\int_{\{|u|\le2n\}}|\nabla u|^{p} \, dx\right)^{\frac{ 1}{p}}.
\end{gather*}
Since $c\in L^{t}(\Omega)$ with $t\geq N/(p-1)$ (see Assumption \eqref{2n0bis})
we get
\[
  \int_{\{|v|\le2n\}}|c(x)|^{\frac{p}{p-1}}(1+| v|)^{{p}} \leq
  \Big(\int_{\Omega} |c|^{N/(p-1)}\, dx\Big)^{p/N}
 \| 1+ T_{2n}(v)\|_{L^{p^{*}}(\Omega)}^{p}
\]
and Poincar\'e-Wirtinger inequality leads to
\[
\int_{\{|v|\le2n\}}|c(x)|^{\frac{p}{p-1}}(1+| v|)^{{p}} \leq C
\Big(\int_{\Omega} |c|^{N/(p-1)}\, dx\Big)^{p/N}
 \big( 1+ \int_{\Omega}|\nabla T_{2n}(v)|^{p} \,dx\big),
\]
where $C>0$ is independent of $n$ and $k$. It follows that
\begin{multline*}
\left | \int_\Omega h'_n(u)h_n(v)T_k(u-v)\Phi (x,  v)\cdot\nabla u\, dx\right |
\le Ck \|c\|_{L^{N/(p-1)}(\Omega)}\\
{}\times \Big( \frac{1}{n}+ \frac{1}{n}\int_{\{|v|\le2n\}}|\nabla v|^{p} \,dx\Big)^{\frac{p-1}{p}}
\Big(\frac1n\int_{\{|u|\le2n\}}|\nabla u|^{p} \, dx\Big)^\frac{1}{p}.
\end{multline*}
Therefore  \eqref{cond_ren} leads to
\begin{equation}\label{step1_3_ren}
\lim_{n\to+\infty}\left| \int_\Omega h'_n(u)h_n(v)T_k(u-v)\Phi(x,  v)\cdot\nabla u\, dx\right|=0,
\end{equation}
and then \eqref{step1_4_ren} holds. We observe that \eqref{step1_5_ren} is obtained by
analogous argument.

Then by $\eqref{iniz_ren}$, \eqref{step1_4_ren}, \eqref{step1_5_ren}, using the
assumptions on the strong monotonicity on the operator \eqref{mon}, the local
Lipschitz condition on $\Phi$ \eqref{lipphi} with $\tau$ which satisfies
\eqref{tau2bis} and Young inequality we get
 \begin{gather}\label{iniz2_ren}
\beta \int_\Omega h_n(u)h_n(v)\frac{|\nabla T_k(u-v)|^2}{(|\nabla u|+|\nabla v|)^{2-p} }\, dx \le 2 \omega_k(n) \\
\notag    + \frac{2k^2}{\beta}    \int_{\{0<|u-v|<k\}}  h_n(u)h_n(v) |c(x)|^2 (1+|u|+|v|)^{2\tau} (|\nabla u|+|\nabla v|)^{2-p}\, dx  , 
\end{gather}
where $\disp\lim_{n}\omega_k(n)=0$. 

We now prove that
\begin{equation}\label{step1_c_ren}|c(x)|^2 (1+|u|+|v|)^{2\tau} (|\nabla u|+|\nabla v|)^{2-p}\in L^1(\Omega),
\end{equation}
so that we can pass to the limit in \eqref{iniz2_ren} as $n\to +\infty$. 
By H\"older inequality  we get 
\begin{gather}\label{step1_int_c_ren}
  \int_{\Omega}   h_n(u)h_n(v) |c(x)|^2 (1+|u|+|v|)^{2\tau} (|\nabla u|+|\nabla v|)^{2-p}\, dx   \\
  \notag  \le \left( \int_{\{|u|<2n,\,|v|<2n\}} h_n(u)h_n(v) |c(x)|^{t}\, dx \right)^{\frac{2 }{t}}\\
  \notag \times 
  \left( \int_{\{|u|<2n,\,|v|<2n\}} h_n(u)h_n(v)  (1+|u|+|v|)^{\nu}\, dx \right)^{\frac{2\tau}{\nu}}\\
  \notag  \times\left( \int_{\{|u|<2n,\,|v|<2n\}} h_n(u)h_n(v) {(|\nabla u|+|\nabla v|)^{\mu}}\, dx \right)^{\frac{ 2-p }{\mu }}
\end{gather}
with
\begin{equation}
  \frac1t+\frac{2\tau}{\nu}+\frac {2-p}{\mu}\le 1, \,\, \nu <\frac {N(p-1)}{N-p}, \,\, \mu <\frac {N(p-1)}{N-1},
\end{equation}
This choice   is possible since  \eqref{tau2bis} holds and in view of
\eqref{ermk4} and \eqref{ermk3} of Proposition \ref{prop2.4} we have
\begin{gather*}
  (1+|u|+|v|)^{\nu} \in L^{1}(\Omega),\quad (|\nabla u|+|\nabla v|)^{\mu} \in L^{1}(\Omega).
\end{gather*}
Passing to the limit as $n$ goes to $+\infty$, assumption \eqref{2n0bis} on $c$
and Fatou Lemma yield that
\[
|c(x)|^2 (1+|u|+|v|)^{2\tau} (|\nabla u|+|\nabla v|)^{2-p}\in L^1(\Omega).
\]
Then we can pass to the limit as $n\to +\infty$ in \eqref{iniz2_ren}, and
dividing  \eqref{iniz2_ren} by $k^{2}$ and using Fatou Lemma we get 
\begin{gather}\label{iniz3_ren}
\frac{1}{k^{2}} \int_\Omega \frac{|\nabla T_k(u-v)|^2}{(|\nabla u|+|\nabla v|)^{2-p} }\, dx\\
 \notag  
\le \frac{1}{\beta^2} \int_{\{0<|u-v|<k\}} |c(x)|^2 (1+|u|+|v|)^{2\tau} (|\nabla u|+|\nabla v|)^{2-p} \, dx\,.
\end{gather}
Recalling that $\chi_{\{0<|u-v|<k\}}$ converges to 0 a.e. as $k$ goes to zero,
 Lebesgue
dominated Theorem and \eqref{step1_c_ren} allow one to conclude that
\eqref{step1fin_ren} holds.
\medskip

\noindent {\sl Step 2.} We prove that either
\begin{equation*}  
   \begin{cases}
       u=v    &\quad \text{a.e. in } \Omega, \\
  u<v    &\quad \text{a.e. in } \Omega,\\
u>v    &\quad \text{a.e. in } \Omega.
 \end{cases}
\end{equation*}
Observe that  for $k<n$ 
\[
  h_n(u)\frac{T_k(u-v)}{k}= h_n(u)\frac{T_k(T_{3n}(u)-T_{3n}(v))}{k} \in
  L^{\infty}(\Omega)\cap W^{1,p}(\Omega).
\]
Then by Poincar\' e-Wirtinger inequality, we get
\begin{gather}\label{PWTkh_n}
\int_{\Omega}\left | h_n(u)  \frac{T_k(u-v)}{k}  -\med\left (h_{n}(u)
\frac{T_{k}(u-v)}{k} \right) \right |^p\, dx     \\
\notag \\
\notag   \le C\int_{\Omega} \left|\nabla \left (h_{n}(u)
\frac{T_{k}(u-v)}{k} \right)\right |^{p}dx .
\end{gather} 
Let us evaluate the integral at the right-hand side. We show that it goes to zero first as $k\to 0$ and then as $n\to +\infty$.

Since 
\begin{gather}\label{gradTkh_n}
\disp \nabla \left (h_{n}(u)
\frac{T_{k}(u-v)}{k} \right)\\
\notag
= h'_n(u)\nabla u \frac{T_k(u-v)}{k}+h_n(u)  \frac{\nabla T_k(u-v)}{k} \quad \text{ a.e. in }\Omega,
\end{gather}
 $\disp \left| \frac{T_{k}(u-v)}{k} \right |\le 1$ and $h'_n(u)\le\frac1n$,
we get
\begin{gather}\label{gradTkh_n_Lp}
\int_{\Omega} \left|\nabla \left (h_{n}(u)
\frac{T_{k}(u-v)}{k} \right)\right |^{p}dx    \\ 
\notag\\
\notag    \le  
\frac{1}{n^p}\int_{\Omega}\left |\nabla T_{2n}(u)\right |^p dx + 
\frac{1}{k^p} \int_{\Omega}h_n(u)^p \left |    {\nabla T_k(u-v)} \right |^p\, dx     \,.
\end{gather} 
 Let us evaluate the second integral in the right hand side of \eqref{gradTkh_n_Lp}. By H\"older inequality we obtain 
 \begin{gather*}\label{?}
 \frac{1}{k^p} \int_{\Omega}h_n(u)^p
\left |    {\nabla T_k(u-v)} 
\right |^p\, dx   \\ 
\notag\\
\notag    \le  
\left (\frac1{k^2} \int_{\Omega} \frac{\left | \nabla T_k(u-v)\right |^2}{{(|\nabla u|+|\nabla v|)^{2-p }}  }\, dx  \right)^{\frac p2} 
\left (\int_{\Omega}h_n(u)^{\frac{2p}{2-p}} \chi_{\{|u-v|<k\}} 
(|\nabla u|+|\nabla v|)^p\, dx  \right)^{\frac{2-p}{2}} 
 \\ 
\notag\\
\notag    \le C\left (\frac1{k^2} \int_{\Omega} \frac{\left | \nabla T_k(u-v)\right |^2}{{(|\nabla u|+|\nabla v|)^{2-p }}  }\, dx  \right)^{\frac p2} 
\left (\int_{\Omega}(|\nabla T_{2n}(u)|^p+|\nabla T_{2n+k}(v)|^p) \, dx  \right)^{\frac{2-p}{2}} ,
\end{gather*} 
that is, if $n$ is fixed, 
\begin{gather*} 
 \frac{1}{k^p} \int_{\Omega}h_n(u)^p
\left |    {\nabla T_k(u-v)} 
\right |^p\, dx   \le  C_n  \left (\frac1{k^2} \int_{\Omega} \frac{\left | \nabla T_k(u-v)\right |^2}{{(|\nabla u|+|\nabla v|)^{2-p }}  }\, dx  \right)^{\frac p2}, 
\end{gather*} 
where  $C_{n}>0$ is a constant depending on $n$ (and independent of $k$). Therefore, by Step 1,  we deduce that
\begin{gather}\label{dastep1}  
 \lim_{k\to 0}\frac{1}{k^p} \int_{\Omega}h_n(u)^p
\left |    {\nabla T_k(u-v)} 
\right |^p\, dx =0\,.
\end{gather} 
Since $\disp \left| \frac{T_{k}(u-v)}{k} \right | $ converges to $\sign (u-v)$
in $L^\infty(\Omega)$ weak-$*$, we deduce from $\eqref{gradTkh_n}$ and
$\eqref{dastep1}$ that for fixed $n$, as $k\to 0$
\[
\disp \nabla \left (h_{n}(u)
\frac{T_{k}(u-v)}{k} \right)\longrightarrow h'_n(u)\nabla u \, \sign (u-v)\,, \,
\, \text{in} \,\, (L^p(\Omega))^N.
\]
We now pass to the limit as $n\to +\infty$. By the definition of $h_n$ we have
$$ 
\int_{\Omega}\left |h'_n(u)\nabla u\right |^p dx \le 
\frac{1}{n^p}\int_{\Omega}\left |\nabla T_{2n}(u)\right |^p
dx
$$
so that \eqref{def3} and \eqref{dastep1} lead to
\begin{equation}\notag
\lim_{n\to +\infty}\lim_{k\to 0}  \int_{\Omega}\left |    \nabla \left (h_{n}(u)
\frac{T_{k}(u-v)}{k} \right) \right |^p\, dx =0.
\end{equation}
Therefore  using \eqref{PWTkh_n}, we deduce that
\begin{gather}\label{lim_ren}
\lim_{n\to +\infty}\,\lim_{k\to 0}  \,\int_{\Omega}\left | h_n(u)  \frac{T_k(u-v)}{k}  -\med\left (h_{n}(u)
\frac{T_{k}(u-v)}{k} \right) \right |^p\, dx=0.
\end{gather}

Since $\left| h_n(u) \frac{T_k(u-v)}{k}  \right|\le 1$, we obtain
$$
\left| \med \left(h_n(u) \frac{T_k(u-v)}{k}\right )\right | \le 1\,, \,\,  k>0.
$$
It follows that, up to a subsequence
$$
\lim_{n\to +\infty} \lim_{k\to 0}  \med \left(h_n(u)\frac{T_k(u-v)}{k}\right )=\gamma\,.
$$
for a suitable constant $\gamma \in \R$, $|\gamma | \le 1$. 

On the other hand
since $u$ is finite a.e. we have
$$ \lim_{k\to 0}h_n(u) \frac{T_k(u-v)}{k} =  h_n(u)\, \sign(u-v)\,,\quad
\text{a.e. and}\,L^{\infty}(\Omega) \text{ weak-$*$}, 
$$
$$\lim_{n\to +\infty} h_n(u) \, \sign(u-v)=   \sign(u-v)\,,\quad\text{a.e.
  and}\,L^{\infty}(\Omega) \text{ weak-$*$.} 
$$
Then, up to subsequence, by {\eqref{lim_ren}} we get
$$
\int_\Omega |\hbox{sign }(u-v)-\gamma|^p\, dx=0.
$$
This implies
$$
\gamma=0\,  \qquad \hbox{or}\qquad \gamma=-1\, \qquad \hbox{or}\qquad \gamma=1,
$$
and means that  either
$$
u=v\,, \, \hbox{a.e. in }\Omega\,\, \hbox{or }  \, u<v\,, \, \hbox{a.e. in }\Omega\,\, \hbox{or } 
\, u>v\,, \, \hbox{a.e. in }\Omega\,. 
$$
\medskip

\noindent {\sl Step 3.} We prove that $u<v\,, \, \hbox{a.e. in }\Omega\,\, \hbox{or } 
\, u>v\,, \, \hbox{a.e. in }\Omega\,$ can not occur. 

\noindent We assume that 
\begin{equation}\label{uv_ren}
u>v\,, \quad \hbox{a.e. in }\Omega
\end{equation}
and  we prove that this yields a contradiction.


The arguments used in Step 3 of Theorem \ref{uniq_weak1} allow us to
prove that
\[
  \text {``$u$ and $v$ have the same sign".}
\]

Let us consider the test function
\begin{equation}\label{natural2}
w_{n,k,\delta}=h_n(u)\frac{T_k(u-v)}{k}\left ( \frac{T_\delta(u^+)}{\delta} -\frac{T_\delta(v^-)}{\delta}  \right)\,,
\end{equation}
for fixed $n>0$, $k>0$, $\delta >0$, where
$$
u^+=\max \{ 0\,, u \}\, , \qquad v^-=\max \{ 0\,, - v \}.
$$
Observe that, since for $k<n$ $h_n(u)\frac{T_k(u-v)}{k}\in
  L^{\infty}(\Omega)\cap W^{1,p}(\Omega)$
we have
\[
  w_{n,k,\delta}\in L^{\infty}(\Omega)\cap W^{1,p}(\Omega).
\]

\noindent We now evaluate the gradient of $w_{n, k,\delta}$:
\begin{gather}\label{gradw_ren}
\nabla w_{n,k,\delta}=\nabla \left (h_n(u) \frac{ T_k(u-v)}{k}\right )  \left ( \frac{T_\delta(u^+)}{\delta} -\frac{T_\delta(v^-)}{\delta}  \right)\\
\notag +
h_n (u) \frac{T_k(u-v)}{k}\left (\frac{\nabla u}{\delta} 
\chi_{\{\, 0<u<\delta  \} } - \frac{\nabla v}{\delta} \chi_{\{\, -\delta<v<0  \} }  \right )  \,\, \hbox{a.e. in }\Omega.
\end{gather}
and we study the limit as $\delta\to 0$, $k\to 0$ and then $n\to +\infty $.
We firstly show that $\med(w_{n,k,\delta})=0$.
Let $\eta $ such that $0<\eta<\frac12$.
\begin{gather*}
\{ x\in \Omega\, :\, w_{n,k,\delta}(x)> \eta \}=\{ x\in \Omega\, :\, w_{n,k,\delta}(x)> \eta,\, 0<u<2n \}\\\
\notag \subset \left \{x\in \Omega\, :\,\frac{T_\delta(u^+)}{\delta}>\eta\right \}=\{x\in \Omega\, :\, u^+ >\eta\delta \}.
\end{gather*}
Since $\med (u)=0$, we have
$$
\meas \{ x\in \Omega\, :\, u(x)> \eta\delta \}< \frac{\meas(\Omega)}{2}\,.
$$
It follows that $\forall  \eta<\frac12,$
$$
\meas \{ x\in \Omega\, :\, w_{n,k,\delta}(x)> \eta \}< \frac{\meas(\Omega)}{2}\,,\,\,
$$
which means $\med (w_{n,k,\delta})\le 0$. 

On the other hand since
$$\forall \eta >0\,,\,\, \{ x\in \Omega\, :\, w_{n,k,\delta}(x)> -\eta \}  \supset \{ x\in \Omega\, :\, u\ge 0 \} $$ 
and 
$$
\meas \{ x\in \Omega\, :\, u\ge0 \}\ge \frac{\meas(\Omega)}{2}\,, 
$$
 we deduce that
$$
\meas \{ x\in \Omega\, :\, w_{n,k,\delta}(x)> -\eta \}\ge \frac{\meas(\Omega)}{2}\,,\,\,\forall  \eta>0,
$$
which means $\med (w_{n,k,\delta})\ge 0$.
We can conclude that 
$$\med (w_{n,k,\delta})= 0.$$

 \noindent Then from  Poincar\' e-Wirtinger inequality, by using \eqref{gradTkh_n_Lp} and \eqref{gradw_ren}, we obtain 
\begin{gather}\label{PWw_ren}
\int_\Omega|w_{n,k,\delta}|^p\, dx\le C \int_{\Omega} |\nabla w_{n,k,\delta}|^p\, dx \\
\notag \le C \Big \{ \frac1{n^p}\int_\Omega \left| \nabla T_{2n}(u)\right |^pdx +  \frac1{k^p}\int_\Omega h_n^p(u) \left| \nabla T_{k}(u-v)\right |^pdx \\
\notag 
+  \frac1{\delta ^p k^p}\int_\Omega h_n^p(u) \left|  T_{k}(u-v)\right |^p  
\left (\left| \nabla T_{\delta}(u^+)\right |+ \left| \nabla T_{\delta}(v^-)\right | \right )^p dx
\Big \}\,.
\end{gather}
We now prove that
\begin{equation}\label{lim_d}\lim_{n\to +\infty}\frac{1}{n^{p}}\int_\Omega   \left| \nabla T_{2n}(u)\right |^p dx=0,
\end{equation} 
\begin{equation}\label{lim_c}\lim_{k\to 0}\frac{1}{k^p}\int_\Omega h_n^p(u) \left| \nabla T_{k}(u-v)\right |^p dx=0,
\end{equation} 
\begin{equation}\label{lim_a}
\lim_{\delta\to 0}\frac1{\delta ^p k^p}\int_\Omega h_n^p(u) \left|  T_{k}(u-v)\right |^p  
 \left| \nabla T_{\delta}(u^+)\right |  ^p dx=0,
 \end{equation}
\begin{equation}\label{lim_b}
\lim_{\delta\to 0}\frac1{\delta ^p k^p}\int_\Omega h_n^p(u) \left|  T_{k}(u-v)\right |^p  
\left| \nabla T_{\delta}(v^-)\right |  ^p dx=0.
\end{equation}
Clearly  \eqref{lim_d} is a consequence of \eqref{def3} in
Definition~\ref{defrenorm}. As far as \eqref{lim_c} is concerned, by H\"older
inequality we have
\begin{gather*}
\frac{1}{k^p}\int_\Omega h_n^p(u) \left| \nabla T_{k}(u-v)\right |^p dx \le 
\left(\frac{1}{k^2}\int_{\{0<|u-v|<k\}} \frac{|\nabla u -\nabla v|^2}{ \left (\left| \nabla u\right |+ \left| \nabla v\right | \right )^{2-p}}dx\right)^{\frac p2}\\
\times \left( \int_{\{0<|u-v|<k\}} h_n^{\frac{2p}{2-p}}(u){ \left (\left| \nabla u\right |+ \left| \nabla v\right | \right )^{p}}dx\right)^{\frac12}
\end{gather*}
and in view of the definition of $h_{n}$, if $n$ is fixed, for any $k<1$ we have
\[
  \begin{split}
    \int_{\{0<|u-v|<k\}} & h_n^{\frac{2p}{2-p}}(u){ \left (\left|
          \nabla u\right |+ \left| \nabla v\right | \right )^{p}}dx
    \\ & {} \le
    \int_{\Omega} { \left (\left| \nabla T_{2n}(u)\right |+ \left| \nabla
          T_{2n+1}(v)\right | \right )^{p}}dx.
    \\ & {} \le C_{n},
  \end{split}
\]
where $C_{n}>0$ is a constant depending on $n$ (and independent of $k$).
From \eqref{step1fin_ren} it follows that for any fixed $n>0$ \eqref{lim_c}
holds.

We now turn to \eqref{lim_a} and \eqref{lim_b}.
Observe that 
$$
\frac1{k^p\delta ^p}{\left|  T_{k}(u-v)\right |^p}   \left| \nabla T_{\delta}(u^+)\right |^p 
=
\frac{1}{k^p\delta^p}  \left| T_{k}(u-v)\right |^p
\left| \nabla  u \right |^p  \chi_{\{\, 0<u<\delta  \} }
$$
a.e. in $\Omega$. Since $u>v$ a.e. in $\Omega$ and $\meas \{ x\in \Omega\, :\, u>0,\, v<0 \}=0\,, $ we get 
$$\left| T_{k}(u-v)\right |^p
  \chi_{\{ 0<u<\delta  \} }\le \delta^p,$$
and then 
$$
\frac1{k^p\delta ^p}{\left|  T_{k}(u-v)\right |^p}   \left| \nabla T_{\delta}(u^+)\right |^p 
\le
\frac{1}{k^p}  
\left| \nabla  u\right |^p  \chi_{\{ 0<u<\delta  \} }.
$$
The Lebesgue dominated Theorem gives for fixed $k>0$, 
$$
\frac{1}{k  }|\nabla u|  \chi_{\{  0<u<\delta   \} }\rightarrow 0
\qquad \hbox{strongly in }L^p(\Omega), \text{ as $\delta\rightarrow0$.}
$$
We deduce \eqref{lim_a}.
In analogous way we get \eqref{lim_b}.


By collecting \eqref{lim_a}, \eqref{lim_b}, \eqref{lim_c}, \eqref{lim_d} and \eqref{PWw_ren} we can conclude that 
$$
\lim_{n\to +\infty}\lim_{k\to 0}\lim_{\delta \to 0} \int_\Omega|w_{n,k,\delta}|^p\, dx=0\,,
$$
which gives, via Lebesgue dominated Theorem,  
$$ \left| \hbox{sign }(u-v) \left ( \chi_{\{  u>0  \}} - \chi_{\{  v<0  \}}\right)\right | =0\,.
$$
This implies that $\chi_{\{ u>0  \}}=\chi_{\{  v<0  \}}$ a.e. in $\Omega$; this yields a contradiction since 
we have proved that $u$ and $v$ have the same sign.

 The same arguments yield that we can not have $u<v$ a.e. in $\Omega$. The conclusion follows.

\end{proof}

\begin{proof}[Proof of Theorem \ref{uniq_ren_2}]

Arguing as in the previous theorem we obtain \eqref{iniz_ren} and we proceed by
dividing the proof by steps. The main difference with respect to the proof of
Theorem \ref{uniq_ren_1} is that for $p>2$ we have to control quadratic terms in
$u-v$ (see \eqref{step1_ren_2}) while $u$ and $v$ are solutions to a $p$-growth problem.
\medskip

\noindent {\sl Step 1.}  By passing to the limit in \eqref{iniz_ren} first as
$n\to +\infty$, then as $k\to 0$ this step is to devoted to prove that
\begin{equation} \label{step1_ren_2}
\lim_{k\to 0}\frac{1}{k^2}\int_\Omega  {|\nabla T_k(u-v)|^2}\, dx  =0.  
\end{equation}

\noindent We  pass to the limit in \eqref{iniz_ren} first as $n\to +\infty$, then
as $k\to 0$.  Arguing as in Step 1 of the previous theorem we get that 
\begin{gather*}
  \begin{split}
\lim_{n\to+\infty}\int_\Omega &h'_n(u)h_n(v)T_k(u-v) \\
& \times (\aop(x, \nabla u)+\Phi(x,  u)- \aop(x, \nabla v)-\Phi(x,  v))
\cdot \nabla u\, dx  =0
\end{split}
\\
\begin{split}
\lim_{n\to+\infty}\int_\Omega & h_n(u)h'_n(v)T_k(u-v)\\
& \times (\aop(x, \nabla u)+ \Phi(x,  u)- \aop(x, \nabla v)-\Phi(x,  v))\cdot \nabla
v\, dx=0\,.
\end{split}
\end{gather*}
Then, using the assumptions on the strong monotonicity on the operator \eqref{mon}, the local Lipschitz condition on $\Phi$ \eqref{lipphi} with $\tau$ which satisfies \eqref{tau} and Young inequality we get
\begin{equation*}
  \begin{split}
    \frac\beta2 &\int_\Omega  h_n(u)h_n(v) (1+|\nabla u| + |\nabla v|)^{p-2} {|\nabla T_k(u-v)|^2} \, dx   \\
    \leq {} & \omega_k(n)
   {}+ \frac{k^2}{2\beta}    \int_{\{0<|u-v|<k\}} h_n(u)h_n(v) |c(x)|^2 (1+|u|+|v|)^{2\tau} \, dx  , 
\end{split}
\end{equation*}
where $\disp\lim_{n}\omega_k(n)=0$. We then obtain
\begin{equation}\label{iniz2_ren2}
  \begin{split}
    \frac\beta2 &\int_\Omega  h_n(u)h_n(v) {|\nabla T_k(u-v)|^2} \, dx  
    \leq \omega_k(n)
    \\
& {}+ \frac{k^2}{2\beta}    \int_{\{0<|u-v|<k\}} h_n(u)h_n(v) |c(x)|^2 (1+|u|+|v|)^{2\tau} \, dx  .
\end{split}
\end{equation}
By H\"older inequality and assumptions on the data we get 
 \begin{equation}\label{step1_4_ren2}
   \begin{split}
     \int_{\Omega} &  h_n(u) h_n(v)  |c(x)|^2 (1+|u|+|v|)^{2\tau} \, dx   \\
 \le {}&\left( \int_{\{x\in\Omega : |u|<2n,\,|v|<2n\}} h_n(u)h_n(v) |c(x)|^{t}\, dx \right)^{\frac{2}{t}}\\
 & \times \left( \int_{\{x\in\Omega : |u|<2n,\,|v|<2n\}} h_n(u)h_n(v)
       (1+|u|+|v|)^{\nu}\, dx \right)^{\frac{t-2 }{2t}}
   \end{split}
\end{equation}
where $\nu= \frac{2t\tau}{t-2}$. According to the assumption on $\tau$  we have
$$  \frac{2t\tau}{t-2}< \frac{N(p-1)}{N-p}$$
which implies $\disp (1+|u|+|v|)^{\nu}\in L^1(\Omega)$. 
Making use of Fatou Lemma and \eqref{step1_4_ren2} we obtain
\begin{equation}
|c(x)|^2 (1+|u|+|v|)^{2\tau} \in L^1(\Omega).\label{step1_c_ren2}
\end{equation}

We can pass to the limit as $n\to +\infty$ in \eqref{iniz2_ren2}, then using Fatou Lemma we get 
\begin{multline}\label{iniz3_ren2}
\frac{1}{k^{2}} \int_\Omega  {|\nabla T_k(u-v)|^2} \, dx\\
\le \frac{1}{\beta^2} \int_{\{0<|u-v|<k\}} |c(x)|^2 (1+|u|+|v|)^{2\tau} \, dx\,.
\end{multline}
Recalling that $\chi_{\{0<|u-v|<k\}}$ converges to 0 a.e. as $k$ goes to zero
  Lebesgue
dominated Theorem and \eqref{step1_c_ren2} allow one to conclude that
 \eqref{step1_ren_2} holds.
\medskip

\noindent {\sl Step 2.} We prove that either
\begin{equation*}  
   \begin{cases}
       u=v    &\quad \text{a.e. in } \Omega, \\
  u<v    &\quad \text{a.e. in } \Omega,\\
u>v    &\quad \text{a.e. in } \Omega.
 \end{cases}
\end{equation*}
Let us consider the function $ h_n(u)\frac{T_k(u-v)}{k}$ and observe that for $k<n$ 
\[
  h_n(u)\frac{T_k(u-v)}{k}= h_n(u)\frac{T_k(T_{3n}(u)-T_{3n}(v))}{k} \in
  L^{\infty}(\Omega)\cap W^{1,p}(\Omega).
\]
Since $p\geq 2$ the  function $ h_n(u)\frac{T_k(u-v)}{k}$ belongs to
$H^{1}(\Omega)$ by Poincar\' e-Wirtinger inequality we get
\begin{gather}\label{PWTkh_n2}
  \begin{split}
    \int_{\Omega}\Big| h_n(u) &\frac{T_k(u-v)}{k} -\med\big(h_{n}(u)
        \frac{T_{k}(u-v)}{k} \big) \Big|^2\, dx     \\
   &{} \le C\int_{\Omega} \left|\nabla \left (h_{n}(u) \frac{T_{k}(u-v)}{k}
      \right)\right |^{2}dx .
  \end{split}
\end{gather} 
Let us evaluate the integral at the right-hand side. We show that it goes to zero first as $k\to 0$ then as $n\to +\infty $.

Since 
\begin{gather*}  
\nabla \left (h_{n}(u)
\frac{T_{k}(u-v)}{k} \right)= h'_n(u)\nabla u \frac{T_k(u-v)}{k}\\
+h_n(u)  \frac{\nabla T_k(u-v)}{k} \quad \text{ a.e. in }\Omega
\end{gather*}
and $\disp \left| \frac{T_{k}(u-v)}{k} \right |\le 1$ 
we get
\begin{gather}\label{gradTkh_n2}
\int_{\Omega} \left|\nabla \left (h_{n}(u)
\frac{T_{k}(u-v)}{k} \right)\right |^{2}dx    \\  
\notag\\
\notag    \le  
\int_{\Omega}\left |h'_n(u)\nabla u\right |^2 dx 
+\frac{1}{k^2} \int_{\Omega}h_n(u)^2 \left |    {\nabla T_k(u-v)} \right |^2\, dx     \,.
\end{gather} 
It is easy to verify that for fixed $n$, as $k\to 0$
$$ \disp \nabla \left (h_{n}(u)
\frac{T_{k}(u-v)}{k} \right)\longrightarrow  h'_n(u)\nabla u \,  \sign (u-v)\,,
\text{ in } (L^2(\Omega))^N$$
Moreover by the definition of $h_n$
$$ 
\int_{\Omega}\left |h'_n(u)\nabla u\right |^2 dx \le 
\frac{1}{n^2}\int_{\Omega}\left |\nabla T_{2n}(u)\right |^2
dx
$$
so that \eqref{def3}, \eqref{step1_ren_2} and \eqref{gradTkh_n2} lead to
\begin{equation}\notag
\lim_{n\to +\infty}\lim_{k\to 0}  \int_{\Omega}\left |    \nabla \left (h_{n}(u)
\frac{T_{k}(u-v)}{k} \right) \right |^2\, dx =0\,.
\end{equation}
Then, using \eqref{PWTkh_n2}, we deduce
\begin{gather}\label{lim_ren_bis}
\lim_{n\to +\infty}\,\lim_{k\to 0}  \,\int_{\Omega}\left | h_n(u)  \frac{T_k(u-v)}{k}  -\med\left (h_{n}(u)
\frac{T_{k}(u-v)}{k} \right) \right |^2\, dx=0.
\end{gather}
Since $\left| h_n(u) \frac{T_k(u-v)}{k}  \right|\le 1$, we obtain
$$
\left| \med \left(h_n(u) \frac{T_k(u-v)}{k}\right )\right | \le 1\,, \,\,  k>0.
$$
It follows that, up to a subsequence, by  \eqref{lim_ren_bis}
$$
\lim_{n\to +\infty} \lim_{k\to 0}  \med \left(h_n(u)\frac{T_k(u-v)}{k}\right )=\gamma\,.
$$
for a suitable constant $\gamma \in \R$, $|\gamma | \le 1$. 

On the other hand
since $u$ is finite a.e.
$$ \lim_{k\to 0}h_n(u) \frac{T_k(u-v)}{k} =  h_n(u)\, \sign(u-v)\,,\quad
\text{a.e. and}\,L^{\infty}(\Omega) \text{ weak-$*$,} 
$$
$$\lim_{n\to +\infty} h_n(u) \, \sign(u-v)=   \sign(u-v)\,,\quad\text{a.e.
  and}\,L^{\infty}(\Omega) \text{ weak-$*$} 
$$
Then, up to subsequence, by {\eqref{lim_ren_bis}} we get
$$
\int_\Omega |\hbox{sign }(u-v)-\gamma|^2\, dx=0
$$
This implies
$$
\gamma=0\,  \qquad \hbox{or}\qquad \gamma=-1\, \qquad \hbox{or}\qquad \gamma=1\,.
$$
This means that  either
$$
u=v\,, \, \hbox{a.e. in }\Omega\,\, \hbox{or }  \, u<v\,, \, \hbox{a.e. in }\Omega\,\, \hbox{or } 
\, u>v\,, \, \hbox{a.e. in }\Omega\,. 
$$

\medskip
Arguing as in Step 3 of the previous theorem, we can prove that the last two possibilities can not occur.
Then conclusion follows.

\end{proof}

\begin{remark} 
As in the case of weak solutions, the existence of   renormalized solutions hold for a class of more general problems \eqref{pbcompleto}
where $f$ belongs to $ L^1(\Omega)$,  $\Phi$ verifies growth conditions and $\aop(x,r, \xi)$ is a Leray-Lions operator which depends on $x$, $s$ and $\xi$ (see \cite{BGM1}).

\noindent Due to the lack of regularity of $u$ in the $L^1$ case by using the techniques developped in the present paper it seems not possible to obtain uniqueness result when $ \aop$ verifies \eqref{ell_comp}-\eqref{lips_a}.
Let us explain the main obstacle in the case $p=2$ and what kind of stronger assumptions on $\aop$ insures the uniqueness of the renormalized solution.
In view of the proof of Theorem \ref{uniq_ren_1} and Theorem \ref{uniq_ren_2} the only new difficulty when $\aop$ depends on $x,\,r,\,\xi$ is to prove {\sl Step 1} which is when $p=2$
\begin{equation} \label{ren_general_1}
\lim_{k\to 0}\frac{1}{k^2}\int_\Omega  {|\nabla T_k(u-v)|^2}\, dx  =0.  
\end{equation} 
In {\sl Step 2} and {\sl Step 3} the structure of the operator does not play any role. 
Equation \eqref{iniz_ren} in which we pass to the limit first as $n\to +\infty$ and then as $k\to 0$ to derive \eqref{ren_general_1} becomes 
\begin{gather}\label{iniz_ren_general}
 \int_\Omega h_n(u)h_n(v)(\aop(x, u,\nabla u)- \aop(x, v,\nabla v))\cdot \nabla T_k(u-v)\, dx   \\ 
\notag    +  \int_\Omega h_n(u)h_n(v)(\Phi(x,  u)- \Phi(x,  v))\cdot \nabla T_k(u-v)\, dx\\
\notag
 +\int_\Omega h'_n(u)h_n(v)T_k(u-v)(\aop(x, u,\nabla u) \qquad\qquad\qquad\qquad
 \\\notag
\qquad \qquad\qquad\qquad{}+ \Phi(x, u)- \aop(x, v,
  \nabla v)-\Phi(x, v))\cdot \nabla u\, dx
\\
\notag    +  \int_\Omega h_n(u)h'_n(v)T_k(u-v)(\aop(x, u,\nabla u)
\qquad\qquad\qquad\qquad \\
\notag
 \qquad\qquad\qquad\qquad {} + \Phi(x,  u)- \aop(x, v,\nabla v)-\Phi(x,  v))\cdot \nabla v\, dx=0\,.
\end{gather}
Since the operator is pseudo-monotone the main obstacle is the control of the first term of \eqref{iniz_ren_general}.
\begin{gather*}\label{ }
 \int_\Omega h_n(u)h_n(v)(\aop(x, u,\nabla u)- \aop(x, v,\nabla v))\cdot \nabla T_k(u-v)\, dx   \\ 
\notag =
 \int_\Omega h_n(u)h_n(v)(\aop(x, u,\nabla u)- \aop(x, u,\nabla v))\cdot \nabla T_k(u-v)\, dx   \\ 
 \notag +
 \int_\Omega h_n(u)h_n(v)(\aop(x, u,\nabla v)- \aop(x, v,\nabla v))\cdot \nabla T_k(u-v)\, dx   \\ 
\notag \ge 
 \beta \int_\Omega h_n(u)h_n(v)\left|\nabla T_k(u-v)\right|^2 \, dx  \\
 \notag +  \int_\Omega h_n(u)h_n(v)(\aop(x, u,\nabla v)- \aop(x, v,\nabla v))\cdot \nabla T_k(u-v)\, dx  \\ 
 \notag \ge 
 \frac\beta2 \int_\Omega h_n(u)h_n(v)\left|\nabla T_k(u-v)\right|^2 \, dx  \\
 \notag -  \int_{\{|u-v|<k\}} h_n(u)h_n(v)\left|\aop(x, u,\nabla v)- \aop(x, v,\nabla v))\right|^2 \, dx  .
 \end{gather*}
 Passing first as $n\to + \infty$ and then as $k\to 0$  requires to have
 \linebreak $\chi_{\{0<|u-v|<k\}} \left|\aop(x, u,\nabla v)- \aop(x, v,\nabla v))\right|^2\in L^1(\Omega)$. If $\aop$ verifies
 \begin{equation*}\label{lips_a_general}
  |\aop(x,s, \xi)- \aop(x,r, \xi)|\leq  |s-r| |\xi|, 
  \end{equation*}
then 
\begin{equation*}
\chi_{\{0<|u-v|<k\}} \left|\aop(x, u,\nabla v)- \aop(x, v,\nabla v))\right|^2\le k^2|\nabla v|^2
\end{equation*}
and we cannot expect to have $|\nabla v|^2\in L^1(\Omega)$ for $L^1$ data.
However by assuming a stronger control of the Lipschitz coefficient of $\aop (x,\,r,\, \xi)$ with respect to $r$, namely
\begin{equation*}\label{lips_a_p=2}
  |\aop(x,s, \xi)- \aop(x,r, \xi)|\leq  \frac{|s-r| }{(1+|s|+|r|)^\lambda}|\xi|, 
  \end{equation*}
with $\lambda>\frac12$, we have 
\begin{equation*}
\chi_{\{0<|u-v|<k\}} \left|\aop(x, u,\nabla v)- \aop(x, v,\nabla v))\right|^2\le k^2 \frac{|\nabla v|^2}{(1+| v|)^{2\lambda}} 
\end{equation*}
and since $2\lambda >1$, estimate \eqref{ermk2} implies
that
\[
  \chi_{\{0<|u-v|<k\}} \left|\aop(x, u,\nabla u)- \aop(x, v,\nabla
    v))\right|^2\in L^1(\Omega).
\]
It follows that
\[
  \chi_{\{0<|u-v|<k\}} \frac{1}{k^2}\left|a(x, u,\nabla v)- a(x,
    v,\nabla v))\right|^2\to 0, \text{  in } L^1(\Omega).
\]
Since the other terms in \eqref{iniz_ren_general} can be controlled by similar
methods to the one used in Theorem~\ref{uniq_ren_2} we are able to conclude
that \eqref{ren_general_1} holds and then that $u=v$ a.e. in $\Omega$.

We now give the complete version of Theorem \ref{uniq_ren_1}  and \ref{uniq_ren_2} for problem \eqref{pbcompleto}.  As in the weak case  we assume that $\aop(x,r, \xi)$
 is a Carath\'eodory 
function which verifies \eqref{ell_comp}, \eqref{growth_comp}, \eqref{mon_comp}, $f\in L^1(\Omega)$ and $\Phi$ verifies \eqref{growthphi} and \eqref{lipphi}. 

When $1<p<2, $ if $\tau\le p-\frac32$ and if  $\aop(x,r, \xi)$
satisfies 
\begin{equation}\label{lips_a_general}
  |\aop(x,s, \xi)- \aop(x,r, \xi)|\leq  C\frac{|s-r| }{(1+|s|+|r|)^\lambda}\left (|\xi|^{p-1}+|s|^{p-1}+|r|^{p-1}+h(x)\right), 
  \end{equation}
  with $\lambda >\frac12$, $h\ge0$ and $h\in L^{p'}(\Omega)$, then the renormalized solution $u$ with null median of \eqref{pbcompleto} is unique. 
  
When $p\ge2$ if 
$$ \tau\le \frac{N(p-1)}{N-p}\left(\frac12 -\frac{1}{t}\right)
$$
and if $\aop(x,r, \xi)$
satisfies 
\begin{gather}\label{lips_a_generalbis}
  |\aop(x,s, \xi)- \aop(x,r, \xi)|  \\
\notag \leq C\frac{|s-r| }{(1+|s|+|r|)^\lambda}\left (|\xi|^{p-1}+|s|^{\frac{p(N-1)}{2(N-p)} }+|r|^{\frac{p(N-1)}{2(N-p)} }+h(x)\right), 
  \end{gather}
 with $\lambda >\frac12$, $h\ge0$ and $h\in L^2(\Omega)$, then the renormalized solution $u$ with null median of \eqref{pbcompleto} is unique. It is worth noting that \eqref{lips_a_general} and \eqref{lips_a_generalbis} are similar except in the power of $|s|$ and $|r|$ and the regularity of $h$. The main reason is that for $p\ge2$ we use quadratic method for a $p$-growth equation. 
\end{remark}

\section*{Acknowledgement}
Research partially supported by Gruppo Nazionale per l'Analisi Matematica, la Probabilit\`a
e le loro Applicazioni (GNAMPA) of the Istituto Nazionale di Alta Matematica (INdAM),  ``Programma triennale
della Ricerca dellÕUniversit\`a degli Studi di Napoli  ``Parthenope'' - Sostegno alla ricerca individuale 2015-2017''.
This work was done during the visits made by the   third author  to
Laboratoire de Math\'ematiques ``Rapha\"el Salem'' de l'Universit\'e
de Rouen and by the second author to 
 Dipartimento di Matematica e
Applicazioni ``R. Caccioppoli'' of University of Naples 
 Federico II. 
 Hospitality and support of all these
institutions are gratefully acknowledged.

\bibliography{bgm1}
\bibliographystyle{plain}

\end{document}